\documentclass[a4paper,reqno]{amsart}

\usepackage{amsfonts}
\catcode`@=11 % we will access private macros of plain TeX (carefully)
\def\blfootnote{\xdef\@thefnmark{}\@footnotetext}
\catcode`@=12 % at signs are no longer letters

\theoremstyle{plain}
\begingroup
\newtheorem{theorem}{Theorem}[section]
\newtheorem{lemma}[theorem]{Lemma}
\newtheorem{proposition}[theorem]{Proposition}
\newtheorem{corollary}[theorem]{Corollary}
\endgroup

\theoremstyle{definition}
\begingroup
\newtheorem{definition}[theorem]{Definition}
\newtheorem{remark}[theorem]{Remark}

\endgroup

\theoremstyle{remark}
\begingroup
\endgroup

\mathchardef\emptyset="001F

\numberwithin{equation}{section}

\newcommand{\be}{\begin{equation}}
\newcommand{\ee}{\end{equation}}

\newcommand{\R}{{\mathbb R}}

\newcommand{\Rm}{{\R}^n}

\usepackage{color}
\definecolor{dmagenta}{rgb}{0.8,0,0.8}

\definecolor{ddcyan}{rgb}{0,0.6,0.9}
\definecolor{dcyan}{rgb}{0,0.4,0.9}
\definecolor{ddmagenta}{rgb}{0.8,0,0.8}
\definecolor{dred}{rgb}{.8,0,0}

\definecolor{vgreen}{rgb}{0.1,0.5,0.2}

\begin{document}
\title[Gamma-convergence of quadratic functionals]{Gamma-convergence of quadratic functionals perturbed by bounded linear functionals}

\author[Gianni Dal Maso]{Gianni Dal Maso}
\address[Gianni Dal Maso]{SISSA, Via Bonomea 265, 34136 Trieste,
Italy}
\email[Gianni Dal Maso]{dalmaso@sissa.it}

\author[Davide Donati]{Davide Donati}
\address[Davide Donati]{SISSA, Via Bonomea 265, 34136 Trieste,
Italy}
\email[Davide Donati]{ddonati@sissa.it}

\thanks{Preprint SISSA 20/2022/MATE}

\begin{abstract}
Given a bounded open set $\Omega\subset \Rm$, we study sequences of quadratic functionals on the Sobolev space $H^1_0(\Omega)$, perturbed by sequences of bounded  linear functionals. We prove that their $\Gamma$-limits, in the weak topology of $H^1_0(\Omega)$, can always be written as the sum of a quadratic functional, a linear functional, and a non-positive constant. The classical theory of $G$- and $H$-convergence completely characterises the quadratic and linear parts of the $\Gamma$-limit and shows that their coefficients do not depend on $\Omega$. The constant, which instead depends on $\Omega$ and will be denoted by $-\nu(\Omega)$, plays an important role in the study of the limit behaviour of the energies of the solutions. The main result of this paper is that, passing to a subsequence, we can prove that $\nu$ coincides with a non-negative Radon measure on a sufficiently large collection of  bounded open sets $\Omega$. Moreover, we exhibit an example that shows that the previous result cannot be obtained for every bounded open set. The specific form of this example shows that the compactness theorem for the localisation method in $\Gamma$-convergence cannot be easily improved.
\end{abstract}
\medskip
\maketitle

{\bf Keywords: }
 $\Gamma$-convergence,  Localisation method, Elliptic equations, $G$-convergence
 
 \medskip
 
{\bf 2020 MSC: }  35J20, 49J45

\section{Introduction}
The aim of this paper is to complete the analysis of the asymptotic behaviour, as $k\to\infty$, of sequences of  functionals of the form 
\begin{equation*}
    F_k(u)=\frac{1}{2}\int_\Omega A_k\nabla u\cdot\nabla u \,dx-\langle f_k,u\rangle \qquad \text{for } u\in H^1_0(\Omega),
\end{equation*}

\noindent where $\Omega\subset\Rm$ is a bounded open set, $A_k$ are $n\times n$ symmetric matrices of $L^\infty(\Omega)$ functions satisfying the usual ellipticity and boundedness conditions, uniformly with respect to $k$, and  $f_k$ is a bounded sequence in the dual $H^{-1}(\Omega)$ of the Sobolev space $H^1_0(\Omega)$.

These functionals, and the asymptotic behaviour of the solutions to their Euler-Lagrange equations 
\begin{equation} \label{eq:problemadacambiare} \begin{cases}
 -\textup{div}(A_k\nabla u_k)= f_k  \,\,\, &\text{in } \Omega,\\
 u_k=0 \, &\text{on }  \partial \Omega,\\
\end{cases}
 \end{equation} 

\noindent have been extensively studied between the late 60s and the 80s. This investigation led to the notion of $G$-convergence of the matrices $A_k$ in the symmetric case (see \cite{DeGiorgi1973,Spagnolo1976,Spagnolo1968}) and of $H$-convergence in the general case (see \cite{Murat1978,Murat1997}). These tools have been widely used in homogenisation problems, where, in the periodic case, the limit matrix can be obtained by solving some auxiliary problems in the periodicity cell (see \cite{Bahvalov1974,Lions1978,Cioranescu1999,Jikov1994,Marchenko2006,Oleuinik1992,Pankov1997}). 

The classical theory of $H$-convergence completely characterises the asymptotic behaviour of the solutions to \eqref{eq:problemadacambiare}, even when the matrices $A_k$ are not symmetric.
Indeed, since $f_k$ is bounded in $H^{-1}(\Omega)$, it can be written as  $f_k=\textup{div}(h_k)+g_k$, where, up to a subsequence, $h_k\in [L^2(\Omega)]^n$ weakly converges to  $h$ in $[L^2(\Omega)]^n$ and $g_k$ weakly converges to $g$ in $L^2(\Omega)$. Therefore \eqref{eq:problemadacambiare} can be equivalently rewritten as
 \begin{equation} \label{rewritten}
\begin{cases}
 -\textup{div}(A_k\nabla u_k+h_k)= g_k  \,\,\,&\text{in } \Omega,\\
 u_k=0  &\text{on } \partial \Omega.\\
 \end{cases}
 \end{equation}

In this form, the problem was studied first in the periodic setting in \cite{Francfort83}, and then in a more general setting in \cite{Francfort91} (see also \cite{Tartar2009} for a comprehensive guide to this kind of problems), where it is shown that, up to a not relabelled subsequence, the solutions $u_k$ to problem \eqref{eq:problemadacambiare} converge weakly in $H^1_0(\Omega)$ to the $u$ solution of 
\begin{equation}  \label{eq:problemacambiato} 
\begin{cases}
 -\textup{div}(A\nabla u)= g \,\,\,  &\text{in } \Omega,\\
 u=0    &\text{on }  \partial \Omega,\\
\end{cases}
 \end{equation}
where $g\in H^{-1}(\Omega)$ is explicitly determined by means of a corrector result (see \cite[Lemma 13.3]{Tartar2009}). This formula implies that, in general, $g$ is not a limit point of the sequence $f_k$ (see also \cite[Proposition 1.2]{Francfort83}). 

As a byproduct, these results also show that the term $g$ is local, in the sense that, if $U\subset\Omega$ is an open set, then the  solutions $u^U_k$ to  problem \eqref{eq:problemadacambiare}, with $\Omega$  replaced by $U$, weakly converge in $H^1_0(U)$ to $u^U$, the unique solution to problem \eqref{eq:problemacambiato} with the same $g$ and with $\Omega$ replaced by $U$.

When the matrices $A_k$ are symmetric, the solutions $u^U_k$ and  $u^U$ can be seen as the unique minimisers of the functionals defined for every $v\in H^1_0(U)$ as
\begin{eqnarray}
    &\displaystyle F_k(v,U)=\frac{1}{2}\int_U A_k\nabla{v}\cdot\nabla{v}\,dx-\langle f_k,\tilde{v}\rangle,\label{intro-Fk}\\&\displaystyle
    F_0(v,U)=\frac{1}{2}\int_UA\nabla{v}\cdot\nabla{v}\,dx-\langle g,\tilde{v}\rangle,\nonumber
\end{eqnarray}
where $\tilde{v}$ is the extension to $\Omega$ of $v$ obtained by setting $\tilde{v}=0$ outside $U$, and $g$ is the term appearing in \eqref{eq:problemacambiato}. 

In this paper, we study the asymptotic behaviour of these functionals in the sense of $\Gamma$-convergence. By a general result on the convergence of minimum values (see \cite[Theorem 7.8]{DalMaso1993}), this is useful to study of the limit the energies of the solutions of
\eqref{eq:problemadacambiare}, with $\Omega$ replaced by $U$; indeed, these energies coincide with the values of $F_k(\cdot, U)$ on their minimisers.

Although, in general, the sequence  $F_k(\cdot,U)$ does not $\Gamma$-converge to $F_0(\cdot,U)$, our main result (see Theorem \ref{thm:maintheorem}) shows that there exist a subsequence, not relabelled, and a bounded non-negative Radon measure $\nu$ on $\Omega$ such that
\begin{equation}\label{eq:convergenceintro}
F_k(\cdot,U)\ \Gamma\textup{-converges to } F_0(\cdot,U)-\nu(U) \textup{ in the weak topology of }H^1_0(U),
\end{equation}
for every $U$ in a {\it rich\/} collection of open subsets of $\Omega$ (see Definition \ref{def:rich}).

This implies that the energies $F_k(u_k^U,U)$ converge to $F_0(u^U,U)-\nu(U)$, hence $\nu(U)$ can be interpreted as the asymptotic energy gap between $u^U$ and the sequence $u^U_k$. This measure theoretical property of the energy gap, which, to our knowledge, was never observed in the literature, completes the picture of the asymptotic behaviour of the solutions of \eqref{eq:problemadacambiare}, with $\Omega$ replaced by $U$, and of their energies.

Simple examples, even in dimension one, show that, if $f_k$ does not converge strongly to $f$ in $H^{-1}(\Omega)$, then the measure $\nu$ may be non-trivial, even if $A_k=A$ for every $k$. A very complex example of a non-zero $\nu $ is presented in Proposition 
~\ref{controesempio}, which is proved for different purposes.

Can we obtain \eqref{eq:convergenceintro} for {\it every\/} open set $U\subset\Omega$, passing possibly to a further subsequence? This raises a general question concerning the compactness result (see \cite[Theorem 16.9]{DalMaso1993}) in the localisation method for $\Gamma$-convergence (see \cite[Chapter 16]{DalMaso1993}): is it possible to find a subsequence, not relabelled, such that
\begin{equation}\label{eq:controesempiogenerale}
F_k(\cdot,U)\ \Gamma\textup{-converges in the weak topology of }H^1_0(U)
\end{equation}
for {\it every\/} open set $U\subset\Omega$? Unfortunately, this is not true, as we show at the end of the paper
 (see Corollary~\ref{controesempio2}).
 More precisely, we exhibit a sequence $F_k$ of functionals of the form \eqref{intro-Fk}, with $A_k$ equal to the identity matrix, such that for every subsequence there exists an open set $U\subset\Omega$ such that \eqref{eq:controesempiogenerale} does not hold.
 In particular, this implies that \eqref{eq:convergenceintro} cannot hold for every open set $U\subset\Omega$, showing that the technicalities involving rich collections of open sets cannot be easily simplified.

  To our knowledge, in the general framework of the localisation method this is the first example of the existence, for every subsequence, of exceptional open sets for which $\Gamma$-convergence  does not hold. This shows that \cite[Theorem 16.9]{DalMaso1993} cannot be easily improved. 

\section{Notation and preliminary results}
 For every bounded open set $U\subset\Rm$ the Sobolev space $H^1(U)$ is the space of functions in $L^2(U)$ whose first-order distributional derivatives are in $L^2(U)$. The space $H^{-1}(U)$ is the dual space of $H^1_0(U)$, the closure in $H^1(U)$ of $C^\infty_c(U)$. We will endow $H^1_0(U)$ with the norm
\begin{equation*}
    \| u \|_{H^1_0(U)}=\Big(\int_U\lvert\nabla{u(x)}\rvert^2dx\Big)^{1/2},
\end{equation*}
which, by the Poincaré Inequality, is equivalent on $H^1_0(U)$ to the usual norm of $H^1(U)$. We denote the duality pairing between $H^{-1}(U)$ and $H^1_0(U)$ with $\langle \cdot,\cdot\rangle$. If $U\subset V$ are bounded open subsets of $\Rm$, every $u\in H^1_0(U)$ can be extended to a function of $H^1_0(V)$, still denoted by $u$, by setting $u=0$ in $V\setminus U$.

Throughout the paper, $\Omega$ is a fixed open bounded subset of $\Rm$, $n\geq 1$, while $\alpha$ and  $\beta$ are two  constants satisfying $ 0<\alpha\leq\beta<+\infty$. We denote by $\mathcal{M}_\alpha^\beta(\Omega)$ the space of symmetric matrices $A\in [L^{\infty}(\Omega)]^{n\times n}$ such that
\begin{equation}\label{eq:ellipticitycond}
  \alpha \lvert\xi\rvert^2\leq A(x)\xi\cdot\xi\leq\beta |\xi|^2\quad \text{ for a.e. }x\in \Omega\,\text{ and every }\xi\in\Rm.
\end{equation}

Although our analysis is done in terms of $\Gamma$-convergence, it is useful to present also the notion of $G$-convergence,  introduced by Spagnolo  in 1968 (see \cite{Spagnolo1976,Spagnolo1968}), since we shall use many results in the literature formulated in this language.

\begin{definition}[$G$-convergence]\label{def:Hconv}
Let $A_k,\,A\in \mathcal{M}_\alpha^\beta(\Omega)$. 
The sequence $A_k$ is said to $G$-converge to $A$ in $\Omega$ if and only if for every $f\in H^{-1}(\Omega)$ the solutions $u_k$ of the problems 
   \begin{equation*} 
    \begin{cases}
    -\text{div}(A_k\nabla{u_k})=f\quad\text{ in }\Omega,\\
    u_k\in H^1_0(\Omega),
    \end{cases}
    \end{equation*}
 converge weakly in $H^1_0(\Omega)$ to  the solution $u$ of the problem
    \begin{equation*}        
\begin{cases}
    -\text{div}(A\nabla{u})=f\quad\text{ in }\Omega,\\
    u\in H^1_0(\Omega).
    \end{cases}
\end{equation*}
\end{definition}

On the space  $\mathcal{M}_\alpha^\beta(\Omega)$ the notion of $G$-convergence coincides with the more general notion of $H$-convergence introduced by Murat and Tartar for possibly non-symmetric matrices (see \cite{Murat1978}, or \cite[Proposition 13.6]{Cioranescu1999}).

It is important to recall that $\mathcal{M}_\alpha^\beta(\Omega)$ is sequentially compact with respect to  $G$-convergence.
\begin{theorem}[{\cite[Theorem 2]{Murat1997}}]\label{thm:seqcompactness}
Let $A_k\in\mathcal{M}^\beta_\alpha(\Omega)$. Then there exist an $A\in\mathcal{M}_\alpha^\beta(\Omega)$ and a subsequence $A_{h_k}$ $G$-converging to $A$ in $\Omega$.
\end{theorem}

Given $A_k,\,A\in\mathcal{M}_\alpha^\beta(\Omega) $ and $U\subset\Omega$ open, we denote by $Q_k(\cdot,U) $ and $Q(\cdot,U)$ the  quadratic functionals defined for every $v\in H^1_0(U)$ as 
\begin{align}
  \label{eq:quadraticfunctionalk}  Q_k(v,U)=\frac{1}{2}\int_UA_k\nabla{v}\cdot\nabla{v}\,dx,
    \\ \label{eq:quadraticfunctionalq} Q(v,U)=\frac{1}{2}\int_UA\nabla{v}\cdot\nabla{v}\,dx.
\end{align}

As it was first pointed out in  \cite{DeGiorgi1973}, the $G$-convergence of a sequence $A_k$  can be characterised in terms of  $\Gamma$-convergence of the associated quadratic functionals \eqref{eq:quadraticfunctionalk}. For the main properties of  $\Gamma$-convergence, we refer the reader to \cite{Braides2002,DalMaso1993}.

\begin{theorem}[{\cite[Theorem 24.5]{DalMaso1993}}]
		\label{thm:equivalenceofGammaandG}
Let $A_k$, $A\in\mathcal{M}_\alpha^\beta(\Omega)$. Then $A_k$ $G$-converges to $A$ in $\Omega$ if and only if $Q_k(\cdot,\Omega)$ $\Gamma$-converges to $Q(\cdot,\Omega)$ in the weak topology of $H^1_0(\Omega)$.
\end{theorem}
We conclude this section with a result we will make use of in the following section.
\begin{proposition}[{\cite[Lemma 13.3]{Tartar2009}}]
    \label{prop:Tartarstorica}
    Let $A_k\in\mathcal{M}_\alpha^\beta(\Omega)$ be a sequence $G$-converging to $A$ in $\Omega$, and let $f_k$ be a bounded sequence in $H^{-1}(\Omega)$. Then there exist $g\in H^{-1}(\Omega)$ and not relabelled subsequences of $A_k$ and $f_k$, such that for any $U\subset\Omega$ open the solutions $u_k^U$ of the problem  
 \begin{equation}\label{eq:problemaconfk}
    \begin{cases}
    -\textup{div}(A_k\nabla{u^U_k})=f_k\quad\text{ in } U,\\
    u^U_k\in H^1_0(U),
    \end{cases}
    \end{equation}
 converge weakly in $H^1_0(U)$   to the solution  $u^U$ of 
 \begin{equation}\label{eq:problemacong}
    \begin{cases}
    -\textup{div}(A\nabla{u^U})=g\quad\text{ in }U,\\
    u^U\in H^1_0(U).
    \end{cases}
    \end{equation}
    \end{proposition}

\begin{remark}\label{re:expressionforg}
  An expression for the term $g$ in the statement of Proposition \ref{prop:Tartarstorica} can be determined by means of the corrector matrices associated to $A_k$. We refer the reader to \cite[Chapter 13]{Tartar2009} for more details. This formula implies that in general $g$ is not a limit point of $f_k$. However, from the same formula, it also follows that, if $A_k\to A$  pointwise a.e. in $\Omega$ and  $f_k\rightharpoonup f$ weakly in $H^{-1}(\Omega)$, then  $f=g$. 
\end{remark}

\section{$\Gamma$-convergence of quadratic functionals}

Let $A_k, A\in\mathcal{M}_\alpha^\beta(\Omega)$, with  $A_k$ $G$-converging to $A$ in $\Omega$, and let $U\subset\Omega$ be an open subset. As in the previous section, $u_k^U$  denotes the sequence of solutions to the elliptic problems
\begin{equation}\label{eq:solutioninU}
    \begin{cases}
         -\textup{div}(A_k\nabla{u_k^U})=f_k\quad\textup{in }U,\\
         u_k^U\in H^1_0(U).
    \end{cases}
\end{equation}
  We  denote by $u^U$  the weak limit in $H^1_0(\Omega)$ of $u_k^U$ (whose existence is guaranteed by Proposition \ref{prop:Tartarstorica}, up to a not relabelled subsequence, independent of $U$), and by $g\in H^{-1}(\Omega)$ the right-hand-side of  \eqref{eq:problemacong}.
  
Theorem \ref{thm:equivalenceofGammaandG} and Proposition
 \ref{prop:Tartarstorica} will allow us to compute the $\Gamma$-limit $F(\cdot,U)$ in the weak topology of $H^1_0(U)$ of the functionals $F_k(\cdot,U)$ defined for $v\in H^1_0(U)$ as
\begin{equation}\label{eq:functionalsFk}
    F_k(v,U)=\frac{1}{2}\int_UA_k\nabla{v}\cdot\nabla{v}\,dx-\langle f_k,v\rangle=Q_k(v,U)-\langle f_k,v\rangle.
\end{equation}
As we will see, the $\Gamma$-limit $F(\cdot,U)$ is closely related to the functional $F_0(\cdot, U)$,
defined for $v\in H^1_0(U)$ as 
\begin{equation}\label{eq:functionalF0}
    F_0(v,U)=\frac{1}{2}\int_U A\nabla{v}\cdot\nabla{v}\,dx-\langle g,v\rangle=Q(v,U)-\langle g,v\rangle.
\end{equation}

Before stating the main result of this work, we briefly recall a definition used in the study of increasing set functions .
\begin{definition}[{\cite[Definition 14.10]{DalMaso1993}}]\label{def:rich}
We say that a collection $\mathcal{R}$ of open subsets of $\Omega$ is rich if, for every family  $(U_t)_{t\in(0,1)}$ of open subsets of $\Omega$ such that
\begin{equation*}
     t_1<t_2\Longrightarrow U_{t_1}\subset\subset U_{t_2},
\end{equation*}
 the set $\{t\in (0,1):\,\,U_{t}\notin\mathcal{R}\}$ is at most countable.
\end{definition}

Examples of rich collections of open sets can be found in \cite[Examples 14.11 and 14.12]{DalMaso1993}. On the contrary, the collection of all open balls contained in~$\Omega$ as well as the collection of all open subsets of $\Omega$ with smooth boundary are not rich (this can be seen by considering as $(U_t)_{t\in (0,1)}$ an increasing family of rectangles).  Unfortunately, the statement of the next theorem cannot be simplified by replacing $\mathcal R$ by the collection of all open subsets of $\Omega$, because of the example given in Corollary~\ref{controesempio2}. By taking $U_0$ equal to an open ball in Remark \ref{re:dopoesempio}, we see that in general we cannot replace $\mathcal{R}$ with the collection of all open balls contained in $\Omega$.
\begin{theorem}\label{thm:maintheorem}
Let $A_k\in \mathcal{M}_\alpha^\beta(\Omega)$ be a sequence of  matrices $G$-converging to $A\in\mathcal{M}_\alpha^\beta(\Omega)$ in $\Omega$. Consider the sequence of functionals $F_k(\cdot, U)$ defined by \eqref{eq:functionalsFk} for $U\subset\Omega$ open. Then there exist a not relabelled subsequence, a non-negative bounded Radon measure $\nu$ on $\Omega$, and a rich collection $\mathcal{R}$ of open subsets of $\Omega$ such that for every $U\in \mathcal{R}$ the sequence $F_k(\cdot,U)$ $\Gamma$-converges in the weak topology of $H^1_0(U)$ to the functional defined as
\begin{equation}\label{eq:limitmain}
    F(v,U)=F_0(v,U)-\nu(U)\quad \text{for every }v\in H^1_0(U),
\end{equation}
where $F_0(\cdot,U)$ is the functional defined in \eqref{eq:functionalF0}, with $g$ as in \eqref{eq:problemacong}.
\end{theorem}
By a general result on $\Gamma$-convergence (see \cite[Theorem 7.8]{DalMaso1993}) Theorem \ref{thm:maintheorem} implies that for every $U\in\mathcal R$, if $u_k^U$ denotes the minimiser  of \eqref{eq:functionalsFk} and  $u^U$ denotes the minimiser  of \eqref{eq:functionalF0}, then the energies $F_k(u_k^U,U)$ converge to $F_0(u^U,U)-\nu(U)$, hence $\nu(U)$ can be interpreted as the asymptotic energy gap between the minimisers of  $F_k(\cdot,U)$ and $F_0(\cdot,U)$. The countable additivity of this gap is a relevant property that completes the description of the behaviour of the solutions of \eqref{eq:problemadacambiare}, with $\Omega$ replaced by $U$,  and of their energies.

To prove Theorem \ref{thm:maintheorem} we proceed in several steps. The first one is to compute the $\Gamma$-$\liminf$ and the $\Gamma$-$\limsup$ of $F_k(\cdot,U)$.
To deal with this problem, we introduce two set functions $\nu'$ and  $\nu''$ defined for every open set $U\subset\Omega$ by
\begin{align}
 \label{eq:defnuinf}  -\nu'(U)=\inf_{\substack{v_k\in H^1_0(U)\\v_k\rightharpoonup 0}}\liminf_{\substack{k\to\infty}} F_k(v_k,U),\\
  \label{eq:defnusup}-\nu''(U)=\inf_{\substack{v_k\in H^1_0(U)\\v_k\rightharpoonup 0}}\limsup_{\substack{k\to\infty}} F_k(v_k,U),
\end{align}
where $v_k\rightharpoonup 0$ means that $v_k$ converges to $0$ in the weak topology of $H^1_0(U)$.
It is easy to see that the infimum is attained. It is immediate to check that the set functions $\nu'$ and $\nu''$ are both increasing and bounded, and that $0\leq\nu''\leq\nu'$.

Note that, being the functionals $F_k(\cdot,U)$ equi-coercive in the weak topology of $H^1_0(U)$, the $\Gamma$-convergence of the sequence $F_k(\cdot, U)$ can be characterised sequentially (see  \cite[Theorem 8.17]{DalMaso1993}), so that $-\nu'(U)$ and $-\nu''(U)$ are precisely the value of the $\Gamma$-$\liminf$  and of the $\Gamma$-$\limsup$ of $F_k(\cdot,U)$ at $v=0$. In what follows  the $\Gamma$-$\liminf$  and the $\Gamma$-$\limsup$ of $F_k(\cdot,U)$ in the weak topology of $H^1_0(U)$ will always be denoted by $F'(\cdot,U)$ and by $F''(\cdot,U)$  and, when they coincide, the $\Gamma$-limit will be denoted by  $F(\cdot,U)$.

\begin{proposition}\label{prop:gammalininfgammalimsup}
Let $U\subset\Omega$ be open and let $F_k(\cdot,U)$ be the functionals defined in \eqref{eq:functionalsFk}. Then:
\begin{enumerate}
    \item[(a)]For every $v\in H^1_0(U)$ 
    \begin{equation}\label{eq:gammaliminf}
        F'(v,U)=F_0(v,U)-\nu'(U);
    \end{equation} 
    \item[(b)] For every $v\in H^1_0(U)$ 
    \begin{equation}\label{eq:gammalimsup}
        F''(v,U)=F_0(v,U)-\nu''(U).
    \end{equation} 
\end{enumerate}
 \end{proposition}
\begin{proof}
We prove only \eqref{eq:gammaliminf}, the proof of \eqref{eq:gammalimsup} being analogous. Let $Q_k(\cdot,U)$ and $Q(\cdot,U)$ be the quadratic functionals defined in \eqref{eq:quadraticfunctionalk} and \eqref{eq:quadraticfunctionalq}. We introduce the non-negative constants $\alpha_k,\,\alpha',\,\alpha'',$ and $\alpha$ defined as 
\begin{equation*}
    \alpha_k=Q_k(u_k^U,U),\,\,\,\,\alpha'=\liminf_{k\to\infty}\alpha_k,\,\,\,\, \alpha''=\limsup_{k\to\infty}\alpha_k,\,\,\,\,\alpha=Q(u^U,U).
\end{equation*}
Since $u_k^U$ and $u^U$ are the solutions to \eqref{eq:problemaconfk} and \eqref{eq:problemacong}, it follows that
\begin{eqnarray}
\label{eq:identityLiminf}&
F_k(v,U)+\alpha_k=Q_k(v-u_k^U,U)\quad \textup{ for every }v\in H^1_0(U),\\
&\label{eq:identityLiminf2}F_0(v,U)+\alpha=Q(v-u^U,U)\quad \textup{ for every }v\in H^1_0(U).
\end{eqnarray}
Fix $v\in H^1_0(U)$. Theorem \ref{thm:equivalenceofGammaandG} then implies that $Q_k(\cdot,U)$ $\Gamma$-converges to $Q(\cdot, U)$ in the weak topology of $H^1_0(U)$, so that if $v_k\rightharpoonup v$ weakly in $H^1_0(U)$ the liminf inequality,\eqref{eq:identityLiminf}, and \eqref{eq:identityLiminf2} yield
\begin{align}
\nonumber\liminf_{k\to\infty}F_k(v_k,U)+\alpha''&\geq\liminf_{k\to\infty} (F_k(v_k,U)+\alpha_k)=\liminf_{k\to\infty}Q_k(v_k-u_k^U,U)\\ \label{eq:identityliminfnok}&\geq Q(v-u^U,U)=F_0(v,U)+\alpha,
\end{align}
whence, taking the infimum with respect to all the sequences $v_k$ converging to $v$ weakly in $H^1_0(U)$,
\begin{equation}\label{eq:geqGammaliminf}
    F'(v,U)\geq F_0(v,U)+\alpha-\alpha''.
\end{equation}

To prove the converse inequality, we argue as follows. Fix $v\in H^1_0(U)$ and consider a sequence $z_k\in H^1_0(U)$ converging to $z:=v-u^U$ weakly in $H^1_0(U)$, so that $v_k:=z_k+u_k^U$ converges weakly to $v$. Then \eqref{eq:identityLiminf} yields
\begin{equation*}
    F'(v,U)+\alpha''\leq\limsup_{k\to\infty}(F_k(v_k,U)+\alpha_k)=\limsup_{k\to\infty}Q_k(z_k,U).
\end{equation*}
Since this last inequality holds for any $z_k\rightharpoonup z$ weakly in $H^1_0(U)$ and $Q_k(\cdot,U)$ $\Gamma$-converges to $Q(\cdot,U)$, we obtain that $F'(v,U)\leq Q(z,U)-\alpha''=F_0(v,U)+\alpha-\alpha''$, where we used \eqref{eq:identityLiminf2}. Together with
\eqref{eq:geqGammaliminf} this implies that
\begin{equation*}
    F'(v,U)=F_0(v,U)+\alpha-\alpha''.
\end{equation*}
Evaluating this last expression at $v=0$ one gets $\nu'(U)=\alpha''-\alpha$, concluding the proof of \eqref{eq:gammaliminf}.
\end{proof}

The following result is an immediate consequence of Proposition \ref{prop:gammalininfgammalimsup}.

\begin{corollary}\label{cor:GammaLimit}The sequence 
$F_k(\cdot,U)$ $\Gamma$-converges in the weak topology of $H^1_0(U)$ if and only if $\nu'(U)=\nu''(U)$, and in that case
the $\Gamma$-limit is given by 
\begin{equation*}
    F(v,U)=F_0(v,U)-\nu'(U)=F_0(v,U)-\nu''(U)\quad\textup{for every }v\in H^1_0(U).
\end{equation*}
\end{corollary}

The second step of our analysis is to make sure that the collection $\mathcal{U}:=\{U\subset\Omega:\,\,U\textup{ open, }\nu'(U)=\nu''(U)\}$ is rich and to construct the measure $\nu$ appearing in \eqref{eq:limitmain}. To this aim, we fix a countable dense collection $\mathcal{D}$ of open subsets of $\Omega$ (i.e., we assume that for all pairs of open sets $U,\,W$  such that $U\subset\subset W\subset \Omega $  there exists $V\in\mathcal{D}$ such that $U\subset\subset V\subset\subset W$). By a diagonal argument and by the compactness of $\Gamma$-convergence (see \cite[Corollary 8.12]{DalMaso1993}), we can pass to a not relabelled subsequence such that $F_k(\cdot,U)$ $\Gamma$-converges in the weak topology of $H^1_0(U)$ for every $U\in \mathcal{D}$. Therefore from now on, we  suppose that $\mathcal{U}$ itself is dense.

Finally, we denote by $\nu$ the common inner regular envelope of $\nu'$ and $\nu''$, i.e., the increasing set function defined for $U\subset\Omega$ open by
\begin{equation}\label{nu}
 \nu(U)=\sup\{\nu'(V)\!:\!V\,\textup{open,}\, V\!\subset\subset\! U\}=\sup\{\nu''(V)\!:\!V\,\textup{open,}\, V\!\subset\subset\! U\}.\hskip-9pt
\end{equation}
\begin{remark}\label{re:regular}
 Since $\nu'$ and $\nu''$ are increasing and the collection $\mathcal{U}$ is dense, by \cite[Proposition 14.14]{DalMaso1993}  the collection $\mathcal{R}:=\{U\in\mathcal{U}: \nu(U)=\nu'(U)=\nu''(U)\}$ is rich.
\end{remark}
We now prove that $\nu$ can be extended to a Borel measure on $\Omega$.

\begin{proposition}\label{prop:nuisameasure}
The set function $\nu$ is the restriction of a Borel measure $\mu$ defined on $\Omega$.
\end{proposition}
\begin{proof}
The proof is based on the De Giorgi-Letta Theorem 
(see \cite[Theorem 5.6]{DeGiorgiLetta1977} for the original result and \cite[Theorem 14.23]{DalMaso1993} for the particular statement used in this proof). Since $\nu$ is clearly increasing and inner regular, we only need to prove that it is subadditive and superadditive to conclude that it can be extended to a Borel measure on $\Omega$.

To prove that $\nu$ is superadditive it is enough to show that $\nu''$ is superadditive (see \cite[Proposition 14.18]{DalMaso1993}). Let $V,\,W\subset\Omega$ be open sets  with $V\cap W=\emptyset$. Consider two sequences $v_k\in H^1_0(V)$ and  $w_k\in H^1_0(W)$ such that $v_k\rightharpoonup0$ weakly in $H^1_0(V)$, $w_k\rightharpoonup0$ weakly in $H^1_0(W)$, and
\begin{equation*}
    -\nu''(V)=\limsup_{k\to\infty}F_k(v_k,V),\quad  -\nu''(W)=\limsup_{k\to\infty}F_k(w_k,W).
\end{equation*}
Let $z_k\in H^1_0(V\cup W)$ be  equal to $v_k$ on $V$ and  to $w_k$ on $W$. Then
\begin{equation*}
    -\nu''(V\cup W)\leq\limsup_{k\to\infty}F_k(z_k,V\cup W)\leq-\nu''(V)-\nu''(W),
\end{equation*}
which concludes the proof of superadditivity.

Arguing as in the proof of \cite[Proposition 18.4]{DalMaso1993} we see that the subadditivity of $\nu$ follows from the following property: if $V',\,V,\,W\subset\Omega$ are open sets with $V'\subset\subset V\subset\Omega$, then
\begin{equation}\label{eq:weaksubadditivity}
\nu'(V'\cup W)\leq\nu'(V)+\nu'(W).
\end{equation}

To prove this inequality we argue as follows. Fix a  cut off function $\varphi$ between $V'$ and $V$ (i.e., $\varphi\in C^\infty_c(\Omega)$, supp$(\varphi)\subset V$, $\varphi=1$ in a neighbourhood of $\overline{V'}$, and $0\leq\varphi\leq1$ in $\Omega$) and consider a sequence $z_k\in H^1_0(V'\cup W)$ such that $z_k\rightharpoonup 0$ weakly in $H^1_0(V'\cup W)$ and 
\begin{equation*}
    -\nu'(V'\cup W)=\liminf_{k\to\infty}F_k(z_k,V'\cup W).
\end{equation*}
We define $v_k=\varphi z_k$ and $w_k=(1-\varphi)z_k$, and we observe that $v_k\in H^1_0(V)$, $w_k\in H^1_0(W)$, that $v_k\rightharpoonup 0$ weakly in $H^1_0(V)$ and  $w_k\rightharpoonup 0$ weakly in $H^1_0(W)$. Therefore
\begin{multline}
    -\nu'(V)-\nu'(W)\leq\liminf_{k\to\infty}F_k(v_k,V)+\liminf_{k\to\infty}F_k(w_k,W)\leq\liminf_{k\to\infty}(F_k(v_k,V)\\\label{eq:sommaregolare}+F_k(w_k,W))
    =\liminf_{k\to\infty}(Q_k(v_k,V)+Q_k(w_k,W)-\langle f_k,z_k\rangle).
\end{multline}

We note that, since $\varphi^2\leq \varphi$, we have
\begin{equation*}
 Q_k(v_k,V)\leq\frac{1}{2}\int_V\!\!\varphi A_k\nabla{z_k}\cdot\nabla{z_k}\,dx+\frac{1}{2}\int_V\!\! z_k^2A_k\nabla{\varphi}\cdot\nabla{\varphi}\,dx+\int_V\!\!\varphi z_kA_k\nabla{z_k}\cdot\nabla{\varphi}\,dx,
\end{equation*}
and a similar inequality holds for $Q_k(w_k,W)$, with $V$ and $\varphi$ replaced by $W$ and $1-\varphi$. Since the sequence $z_k$ converges to zero strongly in $L^2(V'\cup W)$ by the Rellich Compactness Theorem, we have
\begin{equation*}
    Q_k(v_k,V)+Q_k(w_k,W)\leq \frac{1}{2}\int_{V'\cup W}A_k\nabla{z_k}\cdot\nabla{z_k}\,dx+\varepsilon_k,
\end{equation*}
with $\varepsilon_k\rightarrow0$ as $k\to\infty$. Therefore \eqref{eq:sommaregolare} implies
\begin{equation*}
   -\nu'(V)-\nu'(W) \leq \liminf_{k\to\infty}F_k(z_k,V'\cup W)=-\nu'(V'\cup W).
\end{equation*}
This proves \eqref{eq:weaksubadditivity} and concludes the proof of the theorem.
\end{proof}

\begin{proof}[Proof of Theorem \ref{thm:maintheorem}]
The result follows from  Corollary \ref{cor:GammaLimit}, Remark \ref{re:regular}, and Proposition \ref{prop:nuisameasure}.
\end{proof}

 The following proposition shows that there may exist an open set $U_0\subset \Omega$ such that $\nu'$ and $\nu''$ are not inner regular at $U_0$, i.e., $\nu(U_0)<\nu''(U_0)\leq\nu' (U_0)$. As a consequence of this,   Corollary \ref{cor:GammaLimit} implies that \eqref{eq:limitmain} cannot hold for $U=U_0$. As usual, we denote by $\delta_{x_0}$ the Borel measure corresponding to the unit mass concentrated at $x_0$.

\begin{proposition}\label{controesempio}
Assume $n\geq 2$ and $A_k=A=I$, where $I$ is the identity matrix. Let $U_0\subset\subset\Omega$ be open and $x_0\in\partial U_0$. Then there exists a sequence $f_k\in L^\infty(\Omega)$ converging to zero weakly in $H^{-1} (\Omega)$ such that, if $\nu'$ and $\nu''$ are the set functions defined by \eqref{eq:defnuinf} and \eqref{eq:defnusup}, then
\begin{enumerate}
    \item[(a)] $\nu'(U)=\nu''(U)=0$ for every open set $U\subset\Omega$ such that $x_0\notin\overline U$;
    \item[(b)] $\nu'(U)=\nu''(U)=1$ for every open set $U\subset\Omega$ such that $x_0\in U$;
    \item[(c)] $\nu'(U_0)=\nu''(U_0)=1$.
\end{enumerate}
Conditions (a) and (b) imply that the collection $\mathcal{U}:=\{U\subset\Omega:\,\,U\textup{ open, }\nu'(U)=\nu''(U)\}$ is rich, hence $\nu$ is well defined by  \eqref{nu}. Moreover,
\begin{enumerate}
    \item[(d)] $\nu=\delta_{x_0}$.
\end{enumerate}
Finally, (a) and (c) give $\nu(U_0)=0<1=\nu''(U_0)=\nu'(U_0)$.
\end{proposition}
To prove Proposition \ref{controesempio} we use the following well known result.
\begin{lemma}\label{lemma:H-1}
Let $U$ be a bounded open subset of $\Rm$, let $f\in H^{-1}(U)$, and let $u_f$ be the unique solution of the problem
\begin{equation}\label{eq:uf}
\begin{cases}
-\Delta u_f=f\quad\text{ in } U,\\
 u_f\in H^1_0(U).
\end{cases}
\end{equation}
Then $\|f\|_{H^{-1}(U)}=\|u_f\|_{H^1_0(U)}$ and
\begin{equation}\label{eq:minuf}
\min_{u\in H^1_0(U)}\big(\tfrac12\|u\|^2_{H^1_0(U)}-\langle f,u\rangle\big)=\tfrac12\|u_f\|^2_{H^1_0(U)}-\langle f,u_f\rangle=-\tfrac12\|f\|^2_{H^{-1}(U)}.
\end{equation}
\end{lemma}
\begin{proof}
Since \eqref{eq:uf} is the Euler-Lagrange equation of the minimum problem in \eqref{eq:minuf}, the first equality is obvious. The second equality in \eqref{eq:minuf} follows from the fact that $\|f\|_{H^{-1}(U)}=\|u_f\|_{H^1_0(U)}$, which is an immedate consequence of the weak formulation of~\eqref{eq:uf}.
\end{proof}
\begin{proof}[Proof of Proposition \ref{controesempio}]
For every $x\in\mathbb{R}^n$ and $r>0$ the open ball with center $x$ and radius $r$ is denoted by $B(x,r)$. 

{\bf Step 1.} We consider a sequence $f_k\in L^\infty(\Omega)$, bounded in $H^{-1}(\Omega)$, such that for every $\varepsilon>0$ there exists $k_\varepsilon>0$ satisfying
\begin{equation}\label{eq:concentration}
\mathrm{supp}f_k\subset B(x_0,\varepsilon) \text{ for every }k\ge k_\varepsilon.
\end{equation}
Passing to a subsequence, not relabelled, we may assume $f_k$ converges weakly in $H^{-1}(\Omega)$ to some $f_0$.
If $\varphi\in C^\infty_c(\Omega)$ vanishes in a neighbourhood of $x_0$, by \eqref{eq:concentration}  we have that $\langle f_k,\varphi\rangle=0 $ for $k$ large enough, hence   $\langle f_0,\varphi\rangle=0$. On the other hand, since $n\ge 2$, every function in $H^1_0(\Omega)$ can be approximated strongly in $H^1_0(\Omega)$ by a sequence of functions in $C^\infty_c(\Omega)$ vanishing near $x_0$. From the previous remark we conclude that $\langle f_0,v\rangle=0$ for every $v\in H^1_0(\Omega)$, hence $f_0=0$. Since this result does not depend on the subsequence, we conclude that 
\begin{equation}\label{eq:fktozeroweakly}
\text{the whole sequence }f_k\text{ converges to }0\text{ weakly in }H^{-1}(\Omega).
\end{equation}

Property (a) follows immediately from \eqref{eq:concentration} and from the definition of $\nu'$ and $\nu''$.

We claim that if $U$ and $V$ are open subsets of $\Omega$, with $x_0\in U\subset V$, then
\begin{equation}
\label{eq:nuprimoequal1}
\nu'(U)=\nu'(V)\quad\text{and}\quad
\nu''(U)=\nu''(V),
\end{equation}
We prove only the first equality, the proof of the other one being analogous. Since $\nu'(U)\leq\nu'(V)$, it remains to prove that $-\nu'(U)\leq-\nu'(V)$. Let $v_k$ be a sequence converging to $0$ weakly in $H^1_0(V)$ such that 
\begin{equation}\label{eq:recoveryV}
    -\nu'(V)=\liminf_{k\to\infty}F_k(v_k,V).
\end{equation}
Consider a function $\varphi\in C^\infty_c(V)$, with $\mathrm{supp}(\varphi)\subset U$, $\varphi=1$ in a neighbourhood of $x_0$, and $0\leq \varphi\leq 1$ in $U$. Then the sequence $u_k:=\varphi v_k$ belongs to $H^1_0(U)$, $u_k$ converges to $0$ weakly in $H^1_0(U)$, and $u_k=v_k$ in a neighbourhood of $x_0$, independent of $k$. Hence
\begin{eqnarray}\label{eq:qualcosa}
    -\nu'(U)\leq \liminf_{k \to \infty}F_k(u_k,U).
\end{eqnarray}
Since, by \eqref{eq:concentration}, $F_k(u_k,U)$ is equal to
\begin{eqnarray}
   \nonumber
    \frac{1}{2}\int_{V}\varphi^2|\nabla v_k|^2\,dx+\int_V\varphi v_k\nabla\varphi\cdot\nabla v_k\, dx+\frac{1}2\int_Vv_k^2|\nabla\varphi|^2\, dx -\int_Vf_kv_k\,dx\,
\end{eqnarray}
for $k$ large enough and  $v_k$ tends to $0$ strongly in $L^2(V)$ by the Rellich Compactness Theorem, while $\nabla v_k$ is bounded in $L^2(V)$, from \eqref{eq:recoveryV} and \eqref{eq:qualcosa} we obtain $-\nu'(U)\leq-\nu'(V)$, concluding the proof of \eqref{eq:nuprimoequal1}.

If $U$ and $V$ are open subsets of $\Omega$, with $x_0\in U$ and $x_0\in V$, then
\begin{equation}
\label{eq:nuprimoequal11}
\nu'(U)=\nu'(V)\quad\text{and}\quad
\nu''(U)=\nu''(V),
\end{equation}
Indeed, by \eqref{eq:nuprimoequal1} we have $\nu'(U\cap V)=\nu'(U)$, $\nu'(U\cap V)=\nu'(V)$, $\nu''(U\cap V)=\nu''(U)$, and  $\nu'(U\cap V)=\nu'(V)$, which give \eqref{eq:nuprimoequal11}

\medskip
{\bf Step 2.} Let us fix $R>0$ such that $B(x_0,2R)\subset\subset\Omega$. In addition to the assumptions of Step 1, suppose that
 \begin{equation}\label{eq:norm fk}
    \lim_{k\to\infty}\|f_k\|^2_{H^{-1}(B(x_0,R))}=2.
\end{equation}
We claim that for every open set $U\subset\Omega$, with $x_0\in U$, we have
\begin{equation}\label{eq:nu=1inU}
   \nu'(U)=\nu''(U)=1.
\end{equation}
By \eqref{eq:nuprimoequal11} it is enough to prove the equalities when $U=B(x_0,R)$. In this case, recalling that $A_k=I$,  by \eqref{eq:defnuinf} we have
\begin{equation}\label{eq:nuprime}
-\nu'(U)=\inf_{\substack{v_k\in H^1_0(U)\\v_k\rightharpoonup 0}}\liminf_{\substack{k\to\infty}} \Big(\frac12\int_U|\nabla v_k|^2dx-\int_Uf_kv_kdx\Big).
\end{equation}
By \eqref{eq:minuf} and \eqref{eq:norm fk} this implies that
\begin{equation}\label{eq:-nuprimegeq}
-\nu'(U)\geq \liminf_{k\to\infty}\Big( -\frac12\|f_k\|^2_{H^{-1}(U)}\Big)=-1.
\end{equation}

On the other hand, by \eqref{eq:fktozeroweakly} the sequence $u_{f_k}$ of the solutions to \eqref{eq:uf} for $f_k$ converge to $0$ weakly in $H^1_0(U)$. Taking $v_k=u_{f_k}$ in \eqref{eq:nuprime}, and using \eqref{eq:minuf} and \eqref{eq:norm fk} again, we obtain
\begin{eqnarray*}
&\displaystyle-\nu'(U)\leq \liminf_{k\to\infty}\Big(\frac12\int_U|\nabla u_{f_k}|^2dx-\int_Uf_ku_{f_k}dx\Big).
\\
&\displaystyle=\liminf_{k\to\infty}\Big( -\frac12\|f_k\|^2_{H^{-1}(U)}\Big)=-1.
\end{eqnarray*}
Together with \eqref{eq:-nuprimegeq} this gives  $\nu'(U)=1$. The same arguments yield $\nu''(U)=1$, concluding the proof of \eqref{eq:nu=1inU}, which gives (b) in the statement.

Equalities (a) and (b), together with Definition \ref{def:rich}, imply that  the collection $\mathcal{U}:=\{U\subset\Omega:\,\,U\textup{ open, }\nu'(U)=\nu''(U)\}$ is rich, hence $\nu$ is well defined by  \eqref{nu}. Moreover (a), (b), and  \eqref{nu} yield $\nu=\delta_{x_0}$, concluding the proof of (d).

\medskip
{\bf Step 3.} In order to obtain also property (c) of the statement, we now construct a particular sequence $f_k$ in $L^\infty(\Omega)$, bounded in $H^{-1}(\Omega)$, and  satisfying \eqref{eq:concentration} and \eqref{eq:norm fk}. 
We begin by introducing  an auxiliary sequence of functions $g_k$ defined on $\Rm$ and  supported on balls centered at $0$. The sequence $f_k$ will be then obtained as a suitable translation of these functions.

We fix $R>0$ as in Step 2 and a sequence $0<r_k<R$  converging to $0$. The sequence $g_k$ is defined by $g_k=c_k\chi_{B(0,r_k)}$, where $\chi_{B(0,r_k)}$ is the characteristic function of $B(0,r_k)$ and $c_k=\sqrt2/\lvert\rvert \chi_{B(0,r_k)}\lvert\rvert_{H^{-1}(B(0,2R))}$, so that $\|g_k\|^2_{H^{-1}(B(0,2R))}=2$. The argument used at the beginning of Step 1 shows that $g_k$ converges to $0$ weakly in $H^{-1}(B(0,2R))$.

We claim that there exists a sequence $R_k$ converging to $0$ such that $0<r_k<R_k<R$ and 
\begin{equation}\label{eq:new radius}
\|g_k\|^2_{H^{-1}(B(0,R_k))}\to 2.
\end{equation}

Since $g_k$ converges to $0$ weakly in $H^{-1}(B(0,2R))$, the sequence $u_{g_k}$ of the solutions to \eqref{eq:uf} with $U=B(0,2R)$ and $f=g_k$ converge to $0$ weakly in $H^1_0(B(0,2R))$ and strongly in  $L^2(B(0,2R))$ by the Rellich Compactness Theorem.

Let us define 
\begin{equation}\label{eq:defRk}
    R_k:=r_k+\frac{1}{k}+\Big(\int_{B(0,2R)}u_{g_k}^2\, dx\Big)^{1/4}.
\end{equation}
Since $r_k\to 0$ and $u_{g_k}\to 0$ strongly in $L^2(B(0,2R))$, we have that $R_k\to 0$ and $r_k<R_k<R$ for $k$ large enough. Moreover, it follows from \eqref{eq:defRk} that 
\begin{equation}\label{eq:stimapezzetto}
\lim_{k\to\infty}\frac{1}{R_k-r_k}\Big(\int_{B(0,2R)}u_{g_k}^2\,dx\Big)^{1/2}=0.   
\end{equation}

 Consider a sequence $\varphi_k\in C^\infty_c(B(0,2R))$ with $\text{supp}(\varphi_k)\subset B(0,R_k)$, $\varphi_k=1$ on $B(0,r_k)$, $0\leq\varphi_k\leq 1$ on $B(0,2R)$, and  such that $\lvert \nabla{\varphi_k}\rvert\leq2/(R_k-r_k)$. Let $w_k:=\varphi_k u_{g_k}$. Since $w_k\in H^1_0(B(0,R_k))$, by \eqref{eq:minuf} we have
 \begin{equation}\label{eq:leqnormaRk}
     -\tfrac{1}{2}
\|g_k\|^2_{H^{-1}(B(0,R_k))}\leq\tfrac{1}{2}\|w_k\|^2_{H^1_0(B(0,R_k))}-\langle g_k,w_k\rangle. \end{equation}
By a direct computation, setting $U=B(0,2R)$ and $s_k=R_k-r_k$, we obtain
\begin{eqnarray*}
    &\displaystyle\int_{ B(0,R_k)}\lvert\nabla{w_k}\rvert^2\,dx=\int_{U}(\vert\nabla{u_{g_k}}\rvert^2\varphi_k^2+\lvert\nabla{\varphi_k}\rvert^2u_{g_k}^2+2u_{g_k}\varphi_k\nabla{\varphi_k}\cdot\nabla{u_{g_k}})\,dx\\
&\displaystyle\leq\int_{U}\vert\nabla{u_{g_k}}\rvert^2\,dx+\frac{4}{s_k^2}\int_Uu_{g_k}^2\,dx+\frac{4}{s_k}\Big(\int_{U}u_{g_k}^2\,dx\Big)^{1/2}\Big(\int_U\vert\nabla{u_{g_k}}\rvert^2 dx\Big)^{1/2}\\
 &\displaystyle\leq\int_U\vert\nabla{u_{g_k}}\rvert^2\, dx+\frac{4}{s_k}\Big(\int_Uu_{g_k}^2\,dx\Big)^{1/2}\Big(\frac{1}{s_k}\Big(\int_Uu_{g_k}^2\,dx\Big)^{1/2}+M\Big),
\end{eqnarray*}
where $M>0$ is such that $\|\nabla u_{g_k}\|_{L^2(U)}\leq M$. This shows that $w_k$ is bounded in $H^1_0(U)$. Since $u_{g_k}\to 0$ strongly in $L^2(U)$, the sequence $w_k$ converges to $0$ strongly in $L^2(U)$, which together with the boundedness in $H^1_0(U)$, implies that 
\begin{equation}\label{eq:wk converges}
w_k \text{ converges to }0 \text{ weakly in } H^1_0(U).
\end{equation}

Recalling that $w_k=u_{g_k}$ on $\text{supp}(g_k)$, the previous estimate for $|\nabla w_k|^2$, together with   \eqref{eq:stimapezzetto} and \eqref{eq:leqnormaRk}, implies that
\begin{eqnarray}
\label{eq:uguale-1}
   \nonumber &\displaystyle \limsup_{k\to\infty}\big( -\tfrac{1}{2}
\|g_k\|^2_{H^{-1}(B(0,R_k))}\big)\leq\lim_{k\to\infty}\big(\tfrac{1}{2}\|u_{g_k}\|^2_{H^1_0(U)}-\langle g_k,u_{g_k}\rangle\big)\\&\displaystyle=
\lim_{k\to\infty} \tfrac{1}{2}\|g_k\|^2_{H^{-1}(U)}=-1,
\end{eqnarray}
where in the first  equality we used \eqref{eq:minuf} and in the second one we used the equality $\|g_k\|^2_{H^{-1}(U)}=2$, which holds by construction.
Since 
$\|g_k\|^2_{H^{-1}(B(0,R_k))}\leq \|g_k\|^2_{H^{-1}(U)}=2$, we obtain that $R
_k$ satisfies \eqref{eq:new radius}.

We now fix a sequence $x_k\in U_0$ converging to $x_0$ such that $B(x_k,R_k)\subset\subset U_0$ for  $k\in\mathbb{N}$ large enough, and  we set $f_k(x):=g_k(x-x_k)$. We observe that $f_k$ satisfies the concentration property \eqref{eq:concentration} and is bounded in $H^{-1}(\Omega)$. By the equality $\|g_k\|^2_{H^{-1}(B(0,2R))} =2$ and by \eqref{eq:new radius} we also have 
\begin{equation}\label{eq:new norm fk}
\lim_{k\to\infty}\|f_k\|^2_{H^{-1}(B(x_k,R_k))}=\lim_{k\to\infty}\|f_k\|^2_{H^{-1}(B(x_k,2R))} =2.
\end{equation}
Since $B(x_k,R_k)\subset B(x_0,R)\subset B(x_k,2R)\subset \Omega$
for $k$ large enough, we have $\|f_k\|^2_{H^{-1}(B(x_k,R_k))}\le \|f_k\|^2_{H^{-1}(B(x_0,R))}\le \|f_k\|^2_{H^{-1}(B(x_k,2R))}$. Consequently, 
\eqref{eq:new norm fk} implies \eqref{eq:norm fk}. Therefore we can apply to this sequence $f_k$ all results obtained in Steps 1 and 2.

\medskip
\textbf{Step 4.}
We now prove {(c)} for the sequence $f_k$ introduced in the previous step. Since $\nu'(\Omega)=\nu''(\Omega)=1$ thanks to {(b)}, recalling the monotonicity of $\nu' $ and $\nu''$, and the inequality $\nu''\leq \nu'$, it is enough to show that \begin{equation}\label{eq:nu gec 1}
    \nu''(U_0)\geq 1.
\end{equation}

Let $w_k$ be the sequence introduced in Step 3 and let $z_k(x)=w_k(x-x_k)$. Since $w_k$ converge to $0$ weakly in $H^1_0(B(0,2R))$ by \eqref{eq:wk converges} and $z_k\in H^1_0(B(x_k,R_k))\subset H^1_0(U_0)$, we have that $z_k$ converges to $0$ weakly in $H^1_0(U_0)$. By \eqref{eq:defnusup}, we have 
\begin{eqnarray*}
    &\displaystyle -\nu''(U_0)\leq\limsup_{k\to \infty} \Big(\frac{1}{2}\int_{U_0}|\nabla z_k|^2\,dx -\int_{U_0} f_kz_k\,dx\Big)\\
    &\displaystyle =\limsup_{k\to \infty} \Big(\frac{1}{2}\int_{B(0,2R)}|\nabla w_k|^2\,dx -\int_{B(0,2R)} g_kw_k\,dx\Big)\leq - 1,
\end{eqnarray*}
where the last inequality can be obtained arguing as in the proof of \eqref{eq:uguale-1}. This proves \eqref{eq:nu gec 1}, thus concluding the proof of the proposition.
\end{proof}

\begin{corollary}\label{cor:gammaconvergence esempio}
    Under the hypotheses of Proposition \ref{controesempio}, the functional $F_0$ introduced in \eqref{eq:functionalF0} is given by  
    \begin{equation*}
        F_0(u,U):= \frac{1}{2}\int_U|\nabla u|^2\
        \,dx,
    \end{equation*}
for every open set $U\subset \Omega$ and every $u\in H^1_0(U)$ (and we do not need to pass to a subsequence).

Moreover, for every open set $U\subset \Omega$ with $x_0\notin \partial U$ the sequence $F_k(\cdot,U)$ $\Gamma$-converges in the weak topology of $H^1_0(U)$ to the functional $F(\cdot,U):=F_0(\cdot,U)-\delta_{x_0}(U)$.
Finally, $F_k(\cdot,U_0)$ $\Gamma$-converges in the weak topology of $H^1_0(U_0)$ to the functional $F(\cdot,U_0)=F_0(\cdot,U)-1$.
\end{corollary}
\begin{proof}
    The first statement follows from Remark \ref{re:expressionforg}, which shows that in our case we have $g=0$. The other statements are immediate consequences of Corollary \ref{cor:GammaLimit} and Proposition \ref{controesempio}.
\end{proof}
\begin{remark}\label{re:dopoesempio}
    Since $\nu(U_0)=\delta_{x_0}(U_0)=0$, the last statement of Corollary \ref{cor:gammaconvergence esempio} implies that  \eqref{eq:limitmain} does not hold for $U=U_0$.
\end{remark}
%\begin{corollary}\label{controesempio2}
%Let $A_k$, $A$, $U_0$, $f_k$, $\nu'$, and $\nu''$ be as in Proposition~\ref{controesempio} and let $\nu$ be defined by  \eqref{nu}.
%Then for every open set $U\subset\Omega$ with $x_0\not\in\overline{U}$ the sequence 
%$F_k(\cdot,U)$ $\Gamma$-converges to $F_0(\cdot,U)$ in the weak topology of $H^1_0(U)$, while
%$F_k(\cdot,U_0)$ does not $\Gamma$-converges to $F_0(\cdot,U_0)=F_0(\cdot,U_0)-\nu(U_0)$ in the weak topology of $H^1_0(U_0)$.
 %\end{corollary}

%\begin{proof} The first conclusion follows from Corollary \ref{cor:GammaLimit} and Proposition \ref{controesempio}(a). The second one follows from
%Corollary \ref{cor:GammaLimit} and Proposition \ref{controesempio}(b) , taking into account that $\nu(U_0)=0$ by \eqref{nu} and Proposition~\ref{controesempio}(a).
%\end{proof}

The following result shows that  the sequence $F_k$ constructed in the proof of Proposition \ref{controesempio} has no subsequence such that \eqref{eq:limitmain} holds (for the subsequence) for every open subset of $\Omega$. In other words, even if we pass to a subsequence, there always  exists an exceptional open set $U\subset\Omega$ such that $F_k(\cdot,U)$ does not $\Gamma$-converge in the weak topology of $H^1_0(U)$; in particular, this shows that \eqref{eq:limitmain} cannot hold for $U$. More in general,  this shows that, in the localisation method for $\Gamma$-convergence, the usual compactness theorem (see \cite[Theorem 16.9]{DalMaso1993}) cannot be modified so as to obtain the $\Gamma$-convergence on every open set.

In what follows, given a strictly increasing function $\sigma\colon\mathbb{N}\rightarrow\mathbb{N}$, we denote by $\nu'_\sigma(U)$ and $\nu''_\sigma(u)$ the quantities \eqref{eq:defnuinf} and \eqref{eq:defnusup}, with $F_k(\cdot,U)$ substituted by $F_{\sigma(k)}(\cdot,U)$.

\begin{corollary}\label{controesempio2}
Let $n$, $A_k$, $A$, $U_0$, and $x_0$ be as in the statement Proposition~\ref{controesempio} and let $f_k$, $r_k$, $R_k$, and $x_k$ be as in Step 3 of the proof of Proposition~\ref{controesempio}. Then for every strictly increasing function $\sigma\colon \mathbb{N}\rightarrow\mathbb{N}$, there exists an open set $U\subset \Omega$ such that $\nu'_\sigma(U)\neq\nu''_\sigma(U)$. In particular, the sequence $F_{\sigma(k)}(\cdot,U)$ does not $\Gamma$-converge in the weak topology of $H^1_0(U)$.
    \end{corollary}
    \begin{proof}
Without loss of generality, we may suppose that $\sigma(k)=k$ for every $k\in\mathbb{N}$. Since $R_k\to 0$, $x_k\to x_0$, $x_0\in\partial U_0$, and $B(x_k,R_k)\subset\subset U_0$, we can construct recursively a strictly increasing function $\tau\colon\mathbb{N}\rightarrow\mathbb{N}$ such that the closed balls of center $x_{\tau(k)}$ and radius $R_{\tau(k)}$ are pairwise disjoint. We set  $\tau_1(k):=\tau(2k+1)$ and $\tau_2(k):=\tau(2k)$.
Let $C$  be the compact set defined by
\[C:=\bigcup_{i=0}^{\infty}\overline{B(x_{\tau_1(k)},R_{\tau_1(k)})}\cup\{x_0\},\]
and let $U=\Omega\setminus C$. 

Since for any $k\in\mathbb{N}$ the support of $f_{\tau_1(k)}$ is contained in $C$, it follows that $\nu''_{\tau_1}(U)=0$. Recalling that $\nu''\leq\nu''_{\tau_1}$ by \eqref{eq:defnusup},  we obtain $0\leq\nu''(U)\leq\nu''_{\tau_1}(U)=0$, hence $\nu''(U)=0$. 

Conversely, since the closed balls of center $x_{\tau(k)}$ and radius $R_{\tau(k)}$ are pairwise disjoint, for every $k\in\mathbb{N}$ we have $B(x_{\tau_2(k)},R_{\tau_2(k)})\subset\subset U$. Repeating the arguments that lead to \eqref{eq:nu gec 1}, with $\sigma$ replaced by $\tau_2$ and $U_0$ replaced by $U$, we obtain $\nu'_{\tau_2}(U)\geq\nu_{\tau_2}''(U)\geq 1$,  where we have used the obvious inequality $\nu'_{\tau_2}\geq\nu''_{\tau_2}$.  Recalling that $\nu'\geq\nu'_{\tau_2}$ by \eqref{eq:defnuinf} and that $\nu'(\Omega)=1$ by Proposition \ref{controesempio}(b) , we obtain $1\leq \nu'_{\tau_2}(U)\leq\nu'(U)\leq\nu'(\Omega)=1$, where we used also the monotonicity of $\nu'$. This shows that $\nu'(U)=1.$ Since  $\nu''(U)=0$, we have that $\nu''(U)\neq \nu'(U)$. Therefore, $F_k(\cdot,U)$ does not converge in the weak topology of $H^1_0(U)$ by Corollary \ref{cor:GammaLimit}.
\end{proof}

%:References
\noindent \textsc{Acknowledgements.}
 This paper is based on work supported by the National Research Project (PRIN  2017BTM7SN) ``Variational Methods for Stationary and Evolution Problems with Singularities and 
 Interfaces", funded by the Italian Ministry of University and Research. 
The authors are members of the Gruppo Nazionale per 
l'Analisi Matematica, la Probabilit\`a e le loro Applicazioni (GNAMPA) of the 
Istituto Nazionale di Alta Matematica (INdAM). The authors wish to thank Gilles Francfort for some useful information about the proof of the convergence of the solutions of \eqref{rewritten} to the solution of~\eqref{eq:problemacambiato}.

{
}


\frenchspacing
\begin{thebibliography}{99}

\bibitem{Bahvalov1974} Bahvalov N.S.: Averaged characteristics of bodies with periodic structure, Dokl. Akad. Nauk SSSR {\bf 218} (1974), 1046-1048.

\bibitem{Lions1978} Bensoussan A. Lions J.L., Papanicolaou,  Asymptotic analysis for periodic structures, Studies in Mathematics and its Applications, 5, { North-Holland Publishing Co., Amsterdam-New York}, 1978, p. xxiv+700. 

\bibitem{Francfort91} Brahim-Otsmane, S. Francfort, G.A.,  Murat, F.:, Homogenization in thermoelasticity, In: Random media and composites, SIAM, Philadelphia, PA, 1989, p. 13-45.


\bibitem{Braides2002} Braides A.:  $\Gamma$-convergence for beginners, Oxford Lecture Series in Mathematics and its Applications, vol. 22, Oxford University Press, Oxford, 2002, p. xii+218.

     
\bibitem{Cioranescu1999} Cioranescu D.,  Donato P.:  An introduction to homogenization,
Oxford Lecture Series in Mathematics and its Applications, vol. 17, The Clarendon Press, Oxford University Press, New York, 1999, p. x+262.

\bibitem{DalMaso1993} Dal Maso G.:  An introduction to {$\Gamma$}-convergence,
Progress in Nonlinear Differential Equations and their Applications, vol. 8 , Birkh\"{a}user Boston Inc., Boston, 1993, p. xiv+340.

\bibitem{DeGiorgiLetta1977}
De Giorgi E., Letta G.: Une notion g\'{e}n\'{e}rale de convergence faible pour des fonctions croissantes d'ensemble,
Ann. Scuola Norm. Sup. Pisa Cl. Sci.(4) {\bf 4} (1977), 61-99.

\bibitem{DeGiorgi1973} De Giorgi E., Spagnolo S.: Sulla convergenza degli integrali dell'energia per operatori ellittici del secondo ordine, Boll. Un. Mat. Ital. (4) {\bf 8} (1973), 391-411.

\bibitem{Francfort83}
Francfort, G.A.: Homogenization and Linear Thermoelasticity,
SIAM J. Math. Anal. {\bf 14}  (1983), 696-708.

\bibitem{Jikov1994} Jikov, V. V., Kozlov, S. M., Ole\u{\i}nik, O. A.: Homogenization of differential operators and integral functionals,  Springer-Verlag, Berlin, 1994, p. xii+570.

\bibitem{Marchenko2006}
Marchenko V.A, Khruslov E.Y.:  Homogenization of partial differential equations, Progress in Mathematical Physics, vol. 46, Birkh\"{a}user Boston Inc., Boston, 2006, p. xiv+398.

\bibitem{Murat1978}
Murat F.:  {$H$}-convergence, Rapport du séminaire d'analyse fonctionnelle et numérique de l'Université d'Alger, 1978.

\bibitem{Murat1997} Murat F., Tartar L.:  {$H$}-convergence, In: Topics in the mathematical modelling of composite materials, Progr. Nonlinear Differential Equations Appl., vol. 31,  Birkh\"{a}user Boston Inc., Boston, 1997, p. 21-43.

\bibitem{Oleuinik1992}Ole\u{\i}nik, O. A., Shamaev, A. S., Yosifian, G. A.: Mathematical problems in elasticity and homogenization,
Studies in Mathematics and its Applications, vol. 26, North-Holland Publishing Co., Amsterdam, 1992,p. xiv+398.

\bibitem{Pankov1997}
Pankov, A.: {$G$}-convergence and homogenization of nonlinear partial differential operators, 
Mathematics and its Applications, vol. 422, Kluwer Academic Publishers, Dordrecht, 1994, p. xiv+249.

\bibitem{Spagnolo1976}
Spagnolo S.:  Convergence in energy for elliptic operators, In:
Numerical solution of partial differential equations, {III}
({P}roc. {T}hird {S}ympos. ({SYNSPADE}), {U}niv. {M}aryland,
{C}ollege {P}ark, {M}d., 1975). Academic
Press, New York, 1976, p. 469-498.

\bibitem{Spagnolo1968}
Spagnolo S.:  Sulla convergenza di soluzioni di equazioni paraboliche ed ellittiche,
Ann. Scuola Norm. Sup. Pisa (3) {\bf 22} (1968), 571-599.

\bibitem{Tartar2009}
  Tartar L.: The General Theory of Homogenization: A Personalized Introduction,
 Springer Berlin, Heidelberg, 2009, xxii+470. 
\end {thebibliography}
\end{document}